\documentclass[11pt,oneside,english]{amsart}
\usepackage[T1]{fontenc}
\usepackage[latin9]{inputenc}
\usepackage{geometry}
\usepackage{amsthm}
\usepackage{amssymb}

\makeatletter
\numberwithin{equation}{section}
\numberwithin{figure}{section}
  \theoremstyle{plain}
  \newtheorem*{thm*}{\protect\theoremname}
\theoremstyle{plain}
\newtheorem{thm}{\protect\theoremname}
  \theoremstyle{plain}
  \newtheorem{prop}[thm]{\protect\propositionname}
  \theoremstyle{definition}
  \newtheorem{defn}[thm]{\protect\definitionname}
  \theoremstyle{plain}
  \newtheorem{lem}[thm]{\protect\lemmaname}
  \theoremstyle{definition}
  \newtheorem{example}[thm]{\protect\examplename}
  \theoremstyle{remark}
  \newtheorem{rem}[thm]{\protect\remarkname}

\usepackage{amsfonts}
\usepackage{xypic}
\newtheorem{parn}{}[section]

\makeatother

\usepackage{babel}
  \providecommand{\definitionname}{Definition}
  \providecommand{\examplename}{Example}
  \providecommand{\lemmaname}{Lemma}
  \providecommand{\propositionname}{Proposition}
  \providecommand{\remarkname}{Remark}
  \providecommand{\theoremname}{Theorem}
\providecommand{\theoremname}{Theorem}

\begin{document}

\title[Completions of $\mathbb{A}^3$ into Mori fiber spaces]{Explicit biregular/birational geometry of affine threefolds:  completions of $\mathbb{A}^3$ into del Pezzo fibrations and Mori conic bundles}

\author{Adrien Dubouloz}

\address{IMB UMR5584, CNRS, Univ. Bourgogne Franche-Comté, F-21000 Dijon,
France.}

\email{adrien.dubouloz@u-bourgogne.fr}

\author{Takashi Kishimoto}

\address{Department of Mathematics, Faculty of Science, Saitama University,
Saitama 338-8570, Japan.}

\email{tkishimo@rimath.saitama-u.ac.jp}

\thanks{This project was partially supported by ANR Grant \textquotedbl{}BirPol\textquotedbl{}
ANR-11-JS01-004-01, Grant-in-Aid for Scientific Research of JSPS no.
24740003. The work was done during a stay of the first author at Saitama
University and a stay of the second author at the Institute de Math\'ematiques
de Bourgogne (Dijon). The authors thank these institutions for their
hospitality and financial support. }


\begin{abstract}
We study certain pencils $\overline{f}:\mathbb{P}\dashrightarrow\mathbb{P}^{1}$
of del Pezzo surfaces generated by a smooth del Pezzo surface $S$
of degree less or equal to $3$ anti-canonically embedded into a weighted
projective space $\mathbb{P}$ and an appropriate multiple of a hyperplane
$H$. Our main observation is that every minimal model program relative
to the morphism $\tilde{f}:\tilde{\mathbb{P}}\rightarrow\mathbb{P}^{1}$
lifting $\overline{f}$ on a suitable resolution $\sigma:\tilde{\mathbb{P}}\rightarrow\mathbb{P}$
of its indeterminacies preserves the open subset $\sigma^{-1}(\mathbb{P}\setminus H)\simeq\mathbb{A}^{3}$.
As an application, we obtain projective completions of $\mathbb{A}^{3}$
into del Pezzo fibrations over $\mathbb{P}^{1}$ of every degree less
or equal to $4$. We also obtain completions of $\mathbb{A}^{3}$
into Mori conic bundles, whose restrictions to $\mathbb{A}^{3}$ are
twisted $\mathbb{A}_{*}^{1}$-fibrations over $\mathbb{A}^{2}$. 
\end{abstract}
\maketitle

\section*{Introduction}

A threefold Mori fiber space is a mildly singular projective threefold
$X$ equipped with an extremal contraction $\tau:X\rightarrow B$
over a lower dimensional normal projective variety $B$. More precisely,
$X$ has $\mathbb{Q}$-factorial terminal singularities, $\tau$ has
connected fibers, the anti-canonical divisor $-K_{X}$ of $X$ is
ample on the fibers and the relative Picard number $\rho(X/B)=\mathrm{rk}(N_{1}(X))-\mathrm{rk}(N_{1}(B))$
is equal to $1$. These fiber spaces are the possible outputs of Minimal
Model Programs (MMP) ran from rational, or more generally uniruled,
smooth projective threefolds and provide the natural higher dimensional
analogues in this framework of the projective plane and the minimally
ruled surfaces. Noting that rational minimally ruled surfaces $\mathbb{F}_{n}$,
$n\geq2$, $\mathbb{P}^{1}\times\mathbb{P}^{1}$ and $\mathbb{P}^{2}$
are smooth projective completions of the affine plane $\mathbb{A}^{2}$,
it is natural to ask which threefold Mori fiber spaces $\tau:X\rightarrow B$
are projective completions of $\mathbb{A}^{3}$ and, as a first step
towards a potential geometric description of the structure of the
automorphism group $\mathrm{Aut}(\mathbb{A}^{3})$ of $\mathbb{A}^{3}$
from the point of view of the Sarkisov Program \cite{Cor95}, try
to classify them up to birational isomorphisms preserving the inner
open subset $\mathbb{A}^{3}$. 

In the case $\dim B=0$, Fano threefolds of Picard number $1$ containing
$\mathbb{A}^{3}$ have received a lot of attention during the past
decades: a complete classification is known in the smooth case (see
e.g. \cite{Fu93} and the references therein) but the general picture
in the singular case remains elusive. Much less seems to be known
about completions of $\mathbb{A}^{3}$ into ``strict'' Mori fiber
spaces $\tau:X\rightarrow B$, where $\dim B=1,2$. There are two
cases: del Pezzo fibrations when $\dim B=1$ and Mori conic bundles
when $\dim B=2$. Elementary examples of such completions are locally
trivial projective bundles $\tau:\mathbb{P}(\mathcal{O}_{\mathbb{P}^{1}}\oplus\mathcal{O}_{\mathbb{P}^{1}}(m)\oplus\mathcal{O}_{\mathbb{P}^{1}}(n))\rightarrow\mathbb{P}^{1}$
and $\tau:\mathbb{P}(\mathcal{O}_{\mathbb{P}^{2}}\oplus\mathcal{O}_{\mathbb{P}^{2}}(m))\rightarrow\mathbb{P}^{2}$
over $\mathbb{P}^{1}$ and $\mathbb{P}^{2}$, which come respectively
as projective models of linear projections from $\mathbb{A}^{3}$
to $\mathbb{A}^{1}$ and $\mathbb{A}^{2}$. But in general, there
is no reason that the restriction to $\mathbb{A}^{3}$ of the structure
morphism $\tau:X\rightarrow B$ of a completion into a strict Mori
fiber space has general fibers isomorphic to affine spaces. For instance,
since a smooth del Pezzo surface of degree $d\leq3$ with Picard number
$1$ is not rational \cite{Man74}, there cannot exist any completion
of $\mathbb{A}^{3}$ into a del Pezzo fibration $\tau:X\rightarrow B=\mathbb{P}^{1}$
of degree $d\leq3$ whose restriction to $\mathbb{A}^{3}$ is a fibration
with generic fiber isomorphic to the affine plane $\mathbb{A}^{2}$
over the function field of $B$. 

The main purpose of this article is to give examples of ``twisted''
completions of $\mathbb{A}^{3}$ into strict Mori fiber spaces, that
is completions $\tau:X\rightarrow B$ for which the general fibers
of the restriction of $\tau$ to $\mathbb{A}^{3}$ are not isomorphic
to affine spaces. One strategy to construct such examples is to start
from a regular function $f:\mathbb{A}^{3}\rightarrow\mathbb{A}^{1}$
with smooth rational general fibers which extends to a morphism $\tilde{f}':X'\rightarrow\mathbb{P}^{1}$
with smooth general fibers on a smooth projective threefold $X'$
and to a run a relative MMP $\varphi:X'\dashrightarrow X$ over $\mathbb{P}^{1}$.
The rationality of the fibers guarantees that the output $\tilde{f}:X\rightarrow\mathbb{P}^{1}$
is either a del Pezzo fibration or factors through a Mori conic bundle
$\xi:X\rightarrow W$ over a normal projective surface $W$. The main
obstacle is that there is no reason in general that a relative MMP
$\varphi:X'\dashrightarrow X$ preserves the open subset $\mathbb{A}^{3}\subset X'$:
such a process $\varphi$ might contract divisors which are not supported
on the boundary $X'\setminus\mathbb{A}^{3}$, inducing a nontrivial
birational morphism between $\mathbb{A}^{3}$ and its image by $\varphi$
which, in this case is in general again affine, and even worse, small
contractions might occur outside the boundary with the effect that
the image of $\mathbb{A}^{3}$ by $\varphi$ is no longer affine.
As a general fact, understanding the biregular geometry of an affine
threefold via the birational geometry of its projective models requires
to get some effective control on the birational maps appearing in
MMP processes between these models. One solution in our situation
is to consider functions $f:\mathbb{A}^{3}\rightarrow\mathbb{A}^{1}$
extending to fibrations $\tilde{f}':X'\rightarrow\mathbb{P}^{1}$
whose general fibers are already smooth del Pezzo surfaces. Here we
can expect to gain more control on the possible horizontal divisors
contracted by $\varphi$ as well as on its flipping and flipped curves,
and that the output will be in general a del Pezzo fibration, possibly
of higher degree. 

The functions we consider in this article are obtained as restrictions
of pencils $\mathcal{L}$ generated by a smooth del Pezzo surface
$S$ of degree $1$, $2$ or $3$ anti-canonically embedded into a
weighted projective $3$-space $\mathbb{P}$ and an appropriate multiple
$eH$ of a hyperplane $H\in|\mathcal{O}_{\mathbb{P}}(1)|$. Namely,
$\mathbb{P}\setminus H$ is isomorphic to $\mathbb{A}^{3}$ and $f:\mathbb{A}^{3}\rightarrow\mathbb{A}^{1}$
is the restriction of the rational map $\overline{f}:\mathbb{P}\dashrightarrow\mathbb{P}^{1}$
defined by $\mathcal{L}$. For an appropriate class of resolutions
$\sigma:\tilde{\mathbb{P}}\rightarrow\mathbb{P}$ of $\overline{f}:\mathbb{P}\dashrightarrow\mathbb{P}^{1}$
restricting to an isomorphism over $\mathbb{P}\setminus H$ and for
which $\sigma^{-1}(H)$ induces an anti-canonical divisor on the generic
fiber of the induced morphism $\tilde{f}:\tilde{\mathbb{P}}\rightarrow\mathbb{P}^{1}$,
which we call \emph{good resolutions}, we establish that every MMP
$\varphi:\tilde{\mathbb{P}}\dashrightarrow\tilde{\mathbb{P}}'$ relative
to $\tilde{f}$ restricts to an isomorphism between $\tilde{\mathbb{P}}\setminus\sigma^{-1}(H)\simeq\mathbb{A}^{3}$
and its image. The output $\tilde{\mathbb{P}}'$ is then a compactification
of $\mathbb{A}^{3}$ either into a del Pezzo fibration $\tilde{f}':\tilde{\mathbb{P}}'\rightarrow\mathbb{P}^{1}$
or into a Mori conic bundle $\xi:\tilde{\mathbb{P}}'\rightarrow W$
over a certain normal projective surface $q:W\rightarrow\mathbb{P}^{1}$,
and we characterize each possible type of output in terms of the structure
of the base locus of $\mathcal{L}$. Our main result can be summarized
as follows:
\begin{thm*}
Let \textup{$\mathcal{L}\subset|\mathcal{O}_{\mathbb{P}}(e)|$} be
the pencil generated by an anti-canonically embedded smooth del Pezzo
surface $S\subset\mathbb{P}$ of degree $d\in\{1,2,3\}$ and a multiple
of hyperplane $H\in|\mathcal{O}_{\mathbb{P}}(1)|$, let $\sigma:\tilde{\mathbb{P}}\rightarrow\mathbb{P}$
be a good resolution of the corresponding rational map $\overline{f}:\mathbb{P}\dashrightarrow\mathbb{P}^{1}$,
and let $\varphi:\tilde{\mathbb{P}}\dashrightarrow\tilde{\mathbb{P}}'$
be a MMP relative to the induced morphism $\tilde{f}=\overline{f}\circ\sigma:\tilde{\mathbb{P}}\rightarrow\mathbb{P}^{1}$.
Then the induced morphism $\tilde{f}':\tilde{\mathbb{P}}'\rightarrow\mathbb{P}^{1}$
is a projective completion of $\mathbb{A}^{3}$ with $\mathbb{Q}$-factorial
terminal singularities of one of the following types:

a) If $H\cap S$ is irreducible, then $\tilde{f}':\tilde{\mathbb{P}}'\rightarrow\mathbb{P}^{1}$
is a del Pezzo fibration of degree $d$. 

b) If $d=2$ and $H\cap S$ is reducible, then $\tilde{f}':\tilde{\mathbb{P}}'\rightarrow\mathbb{P}^{1}$
is del Pezzo fibration of degree $d+1=3$. 

c) If $H\cap S$ has three irreducible components, then $\tilde{\mathbb{P}}'$
is a Mori conic bundle. 
\end{thm*}
In the case where the output $\tilde{\mathbb{P}}'$ is a Mori conic
bundle $\xi:\tilde{\mathbb{P}}'\rightarrow W$, we establish further
that the restriction of $\xi$ to the inner $\mathbb{A}^{3}$ is a
\emph{twisted} $\mathbb{A}_{*}^{1}$-\emph{fibration} $\xi_{0}:\mathbb{A}^{3}\rightarrow\mathbb{A}^{2}$,
that is, a flat fibration whose generic fiber is a nontrivial form
of the punctured affine line $\mathbb{A}_{*}^{1}$ over the function
field of $\mathbb{A}^{2}$. This contrasts with the situation for
$\mathbb{A}^{2}$ for which no such type of $\mathbb{A}_{*}^{1}$-fibration
can exist, essentially as a consequence of Tsen's theorem and the
factoriality of $\mathbb{A}^{2}$ (see \cite[Lemma 1.7.2]{MiyBook}).
We also provide a geometric interpretation of these fibrations in
terms of the pair $(S,H)$ initially chosen for the construction.

\section{Pencils of del Pezzo surfaces in weighted projective spaces}

Recall that a smooth del Pezzo surface is a smooth projective surface
$S$ whose anti-canonical divisor $-K_{S}$ is ample. The integer
$d=(-K_{S}^{2})\in\{1,\ldots,9\}$ is called the degree of $S$. Every
such surface is either isomorphic to $\mathbb{P}^{1}\times\mathbb{P}^{1}$
or to the blow-up of the projective plane $\mathbb{P}^{2}$ in $9-d$
points in general position \cite{Man74}. Anti-canonical models $\mathrm{Proj}_{\mathbb{C}}(\bigoplus_{m\geq0}H^{0}(S,-mK_{S}))$
of smooth del Pezzo surfaces of degree $\leq3$ are naturally embedded
as hypersurfaces in certain weighted projective spaces. Their properties
are summarized by the following proposition: 
\begin{prop}
\label{prop:Embedded-dP-summary}\indent 

1) Every smooth del Pezzo surface of degree $3$ is isomorphic to
a smooth cubic surface in $\mathbb{P}^{3}$, and conversely every
smooth cubic surface $S$ in $\mathbb{P}^{3}$ is a del Pezzo surface
of degree $3$. For every hyperplane $H\in|\mathcal{O}_{\mathbb{P}^{3}}(1)|$,
$H\mid_{S}$ is a reduced anti-canonical divisor on $S$ whose support
is isomorphic to a plane cubic curve, either irreducible, or consisting
of the union of a smooth conic $C$ and a line $\ell$ intersecting
each other twice, either transversally in two distinct points or tangentially
in a unique point, or consisting of three lines, either in general
position or intersecting each other in a unique point, which is then
an Eckardt point of $S$. 

2) Every smooth del Pezzo surface of degree $2$ is isomorphic to
a smooth quartic hypersurface of the weighted projective space $\mathbb{P}(1,1,1,2)$.
Conversely, every smooth quartic $S$ in $\mathbb{P}(1,1,1,2)$ is
a del Pezzo surface of degree $2$. For every $H\in|\mathcal{O}_{\mathbb{P}(1,1,1,2)}(1)|$,
$H\mid_{S}$ is a reduced anti-canonical divisor on $S$ whose support
is isomorphic either to an irreducible plane cubic curve, or to the
union of two $(-1)$-curves on $S$ intersecting each other twice,
either transversally in two distinct points or tangentially in a unique
point.

3) Every smooth del Pezzo surface of degree $1$ is isomorphic to
a smooth sextic hypersurface of the weighed projective space $\mathbb{P}(1,1,2,3)$,
and conversely every smooth sextic in $\mathbb{P}(1,1,2,3)$ is a
del Pezzo surface of degree $1$. For every $H\in|\mathcal{O}_{\mathbb{P}(1,1,2,3)}(1)|$,
$H\mid_{S}$ is an irreducible and reduced anti-canonical divisor
on $S$ whose support is isomorphic to a plane cubic curve. 
\end{prop}

\subsection{Pencils of del Pezzo surfaces of degree $\leq3$ }

In what follows, given an anti-canonically embedded smooth del Pezzo
surface $S$ of degree $d\leq3$ as in Proposition \ref{prop:Embedded-dP-summary}
above, we use the same notation $\mathbb{P}=\mathrm{Proj}(\mathbb{C}[x,y,z,w])$
to denote the ambient spaces $\mathbb{P}^{3}$, $\mathbb{P}(1,1,1,2)$
and $\mathbb{P}(1,1,2,3)$ according to $d=3$, $2$ and $1$, the
variables $x$, $y$, $z$ and $w$ having degrees $(1,1,1,1)$, $(1,1,1,2)$
and $(1,1,2,3)$ respectively. The degree of $S$ as a hypersurface
of $\mathbb{P}$ is denoted by $e$. It is equal to $3$, $4$ or
$6$ according to $d=3$, $2$ and $1$ respectively. 
\begin{defn}
\label{def:Pencils}Let $S\subset\mathbb{P}$ be a smooth del Pezzo
surface of degree $d\in\{1,2,3\}$ and let $H\in|\mathcal{O}_{\mathbb{P}}(1)|$
be a hyperplane. We denote by $\mathcal{L}\subset|\mathcal{O}_{\mathbb{P}}(e)|$
the pencil generated by $S$ and $eH$ and by $\overline{f}:\mathbb{P}\dashrightarrow\mathbb{P}^{1}=\mathbb{P}(\mathcal{L}^{*})$
the corresponding rational map. 
\end{defn}

\subsubsection{\label{sub:Coordinate-presentation}\noindent }

A member $S_{[\alpha:\beta]}$, $[\alpha:\beta]\in\mathbb{P}^{1}$,
of $\mathcal{L}$ is defined up to a linear transformation of $\mathbb{P}$
by the vanishing of a weighted-homogeneous polynomial $F\in\mathbb{C}[x,y,z,w]$
of degree $e$ of the form 
\[
F=\beta s(x,y,z,w)-\alpha x^{e},
\]
 where $S$ and $H$ are defined respectively by the vanishing of
$s(x,y,z,w)$ and $x$. The scheme-theoretic base locus of $\mathcal{L}$
is equal to the closed subscheme of $\mathbb{P}$ defined by the weighted-homogeneous
ideal $(s(x,y,z,w),x^{e})$ of $\mathbb{C}[x,y,z,w]$. Its support
is equal to $H\cap S$. With this description, the rational map $\overline{f}:\mathbb{P}\dashrightarrow\mathbb{P}^{1}$
coincides with that defined by $[x:y:z:w]\mapsto[s(x,y,z,w):x^{e}]$.
The complement of $H$ is isomorphic to $\mathbb{A}^{3}$ with inhomogeneous
coordinates $Y=x^{-1}y$, $Z=x^{-a}z$ and $W=x^{-b}w$, where $(a,b)=(1,1)$,
$(1,2)$ and $(2,3)$ according to $d=3$, $2$ and $1$ respectively,
and letting $\infty=[1:0]=\overline{f}_{*}(H)\in\mathbb{P}^{1}$,
the restriction of $\overline{f}$ to $\mathbb{P}\setminus H$ coincides
with the regular function 
\[
f:\mathbb{A}^{3}=\mathbb{P}\setminus H\rightarrow\mathbb{A}^{1}=\mathbb{P}^{1}\setminus\{\infty\}\simeq\mathrm{Spec}(\mathbb{C}[\lambda]),\;\left(X,Y,Z\right)\mapsto s(1,Y,Z,W).
\]

The \emph{generic member} $S_{\eta}$ of $\mathcal{L}$, that is,
the closure in $\mathbb{P}_{\mathbb{C}(\lambda)}=\mathrm{Proj}(\mathbb{C}(\lambda)[x,y,z,w])$
of the fiber of $f$ over the generic point $\eta$ of $\mathbb{P}^{1}$,
is isomorphic to the projective surface over $\mathbb{C}(\lambda)$
defined by the vanishing of weighted-homogeneous polynomial $s(x,y,z,w)+\lambda x^{e}\in\mathbb{C}(\lambda)[x,y,z,w]$.
Since $S$ is smooth, it follows from the Jacobian criterion that
$S_{\eta}$ is smooth, hence is a smooth del Pezzo surface of degree
$d$ defined over the function field $\mathbb{C}(\lambda)$ of $\mathbb{P}^{1}$.
This implies in particular that the general member of $\mathcal{L}$
is a smooth del Pezzo surface of degree $d$. Some members of $\mathcal{L}$
can be singular (see Example \ref{ex:singular-examples} below) but
all members of $\mathcal{L}$ except $eH$ are integral schemes:
\begin{lem}
\label{lem:irreducible-members}All members of $\mathcal{L}$ except
$eH$ are irreducible and reduced. \end{lem}
\begin{proof}
We consider each degree $d=3,2,1$ separately. If $d=3$ and $S'\in\mathcal{L}\setminus\{S,3H\}$
is either reducible or non reduced, then one of its irreducible components
is necessarily a hyperplane, say $H'$, which is different from $H$
as $\mathcal{L}$ does not have any fixed component. So $H'\cap S$
is distinct from $H\cap S$, hence is strictly contained in it as
$H\cap S$ coincides with the support of the base locus of $\mathcal{L}$.
This is absurd in view of 1) in Proposition \ref{prop:Embedded-dP-summary}. 

In the case $d=2$, a member $S'\in\mathcal{L}\setminus\{S,4H\}$
which is either reducible or non reduced contains an irreducible component
of degree one or two. In the first case, we would have again a hyperplane
$H'\in|\mathcal{O}_{\mathbb{P}(1,1,1,2)}(1)|$ distinct from $H$
for which $H'\cap S$ is contained in $H\cap S$, which is absurd
by virtue of 2) in Proposition \ref{prop:Embedded-dP-summary}. In
the second case, $S'$ would be the union of two irreducible quadric
hypersurfaces $Q_{1}$ and $Q_{2}$ of $\mathbb{P}(1,1,1,2)$, necessarily
distinct from each other since otherwise every member of $\mathcal{L}$
would be reducible. Since the restriction map 
\[
H^{0}(\mathbb{P}(1,1,1,2),\mathcal{O}_{\mathbb{P}(1,1,1,2)}(2))\stackrel{\sim}{\rightarrow}H^{0}(S,\mathcal{O}_{\mathbb{P}(1,1,1,2)}(2)\mid_{S})\simeq H^{0}(S,\mathcal{O}_{S}(-2K_{S}))
\]
is an isomorphism, both intersections $Q_{i}\cap H$, $i=1,2$ are
strictly contained in $H\cap S$. Indeed, if $Q_{i}\cap H=H\cap S$
then $Q_{i}\mid_{S}=2H\mid_{S}$ and then $Q_{i}=2H$ contradicting
the irreducibility of $Q_{i}$. This implies in turn by virtue of
2) in Proposition \ref{prop:Embedded-dP-summary} that $Q_{i}\mid_{S}$
is supported on a $(-1)$-curve, which is absurd as $Q_{i}\mid_{S}$
has non negative self-intersection. 

Finally, if $d=1$ and $S'\in\mathcal{L}\setminus\{S,6H\}$ is not
integral, then it contains an irreducible component $P$ of degree
$1$, $2$ or $3$. Because of the isomorphisms 
\[
H^{0}(\mathbb{P}(1,1,2,3),\mathcal{O}_{\mathbb{P}(1,1,2,3)}(j))\stackrel{\sim}{\rightarrow}H^{0}(S,\mathcal{O}_{\mathbb{P}(1,1,2,3)}(j)\mid_{S})\simeq H^{0}(S,\mathcal{O}_{S}(-jK_{S})),\quad j=1,2,3,
\]
the same argument as in the previous case implies that $P\cap S$
is strictly contained in $H\cap S$, which is absurd since the latter
is irreducible by virtue of 3) in Proposition \ref{prop:Embedded-dP-summary}. \end{proof}
\begin{example}
\label{ex:singular-examples}

a) The sextics $S_{1}$ and $S_{2}$ in $\mathbb{P}(1,1,2,3)=\mathrm{Proj}_{\mathbb{C}}(\mathbb{C}[x,y,z,w])$
defined respectively by the equations $z^{3}+w^{2}+xy^{5}=0$ and
$z^{3}+w^{2}+x^{2}(x^{3}y+z^{2})=0$ are normal del Pezzo surfaces
with a unique singular point of type $E_{8}$ at $p_{1}=[1:0:0:0]$
and $p_{2}=[0:1:0:0]$ respectively. The general members of the pencils
$\overline{f}_{i}:\mathbb{P}(1,1,2,3)\dashrightarrow\mathbb{P}^{1}$,
$i=1,2$, generated respectively by $S_{1}$ and $6H_{1}$ where $H_{1}=\{x+by=0\}\in|\mathcal{O}_{\mathbb{P}(1,1,2,3)}(1)|$,
$b\in\mathbb{C}$, and $S_{2}$ and $6H_{2}$ where $H_{2}=\{ax+y=0\}\in|\mathcal{O}_{\mathbb{P}(1,1,2,3)}(1)|$,
$a\in\mathbb{C}$, are smooth del Pezzo surfaces of degree $1$. The
intersection $S_{1}\cap H_{1}$ is either a rational cuspidal cubic
if $b=0$ or a smooth elliptic curve otherwise, while $S_{2}\cap H_{2}$
is either a nodal cubic if $a=0$ or a smooth elliptic curve otherwise. 

b) The quartic surface $S=\{w^{2}+yz^{3}+xy^{3}=0\}$ in $\mathbb{P}(1,1,1,2)=\mathrm{Proj}_{\mathbb{C}}(\mathbb{C}[x,y,z,w])$
is a normal del Pezzo surface with a unique singular point of type
$E_{7}$ at $[1:0:0:0]$. The general members of the pencil $\overline{f}:\mathbb{P}(1,1,1,2)\dashrightarrow\mathbb{P}^{1}$
generated by $S$ and $4H$ where $H=\{x+ay+bz=0\}\in|\mathcal{O}_{\mathbb{P}(1,1,1,2)}(1)|$,
$a,b\in\mathbb{C}$ are smooth del Pezzo surfaces of degree $2$.
The intersection of $H$ with $S$ is either a cuspidal cubic if $b=0$
or a smooth elliptic curve otherwise. 

c) The cubic surfaces $S_{1}(\lambda)=\{x^{3}+w(\lambda x^{2}+y^{2}+wz)=0\}$,
$\lambda\in\mathbb{C}$, and $S_{2}=\{xyz+y^{3}+w^{2}z=0\}$ in $\mathbb{P}^{3}$
are normal del Pezzo surfaces respectively with a unique singularity
of type $E_{6}$ at $[0:0:1:0]$ and a pair of singular points $[1:0:0:0]$
and $[0:0:1:0]$ of types $A_{1}$ and $A_{5}$. The general members
of the pencils $\overline{f}_{i}:\mathbb{P}^{3}\dashrightarrow\mathbb{P}^{1}$,
$i=1,2$, generated respectively by $S_{1}(\lambda)$ and $3H_{1}$,
where $H_{1}=\{z=0\}$, and by $S_{2}$ and $3H_{2}$, where $H_{2}=\{x+z=0\}$,
are smooth cubic surfaces. 
\end{example}

\section{Good resolutions and relative MMPs}

In this section, we introduce particular resolutions $\sigma:\tilde{\mathbb{P}}\rightarrow\mathbb{P}$
of the indeterminacies of the rational map $\overline{f}:\mathbb{P}\dashrightarrow\mathbb{P}^{1}$
associated to a pencil as in Definition \ref{def:Pencils} above.
These have the property to restrict to isomorphisms over the open
subset $\mathbb{A}^{3}=\mathbb{P}\setminus H$, and we show that every
MMP process $\varphi:\tilde{\mathbb{P}}\dashrightarrow\tilde{\mathbb{P}}'$
relative to the induced morphism $\tilde{f}=\overline{f}\circ\sigma:\tilde{\mathbb{P}}\rightarrow\mathbb{P}^{1}$
again preserves $\tilde{\mathbb{P}}\setminus\sigma^{-1}(H)\simeq\mathbb{P}\setminus H$,
inducing an isomorphism between $\tilde{\mathbb{P}}\setminus\sigma^{-1}(H)\simeq\mathbb{P}\setminus H$
and $\tilde{\mathbb{P}}'\setminus\varphi_{*}(\sigma^{-1}(H))$.

\subsection{Good resolutions of del Pezzo pencils }

Let $S\subset\mathbb{P}$ be a smooth del Pezzo surface of degree
$d\leq3$, let $\mathcal{L}\subset|\mathcal{O}_{\mathbb{P}}(e)|$
be the pencil generated by $S$ and $eH$ for some $H\in|\mathcal{O}_{\mathbb{P}}(1)|$
and let $\overline{f}:\mathbb{P}\dashrightarrow\mathbb{P}^{1}$ be
the corresponding rational map as in Definition \ref{def:Pencils}.
Similarly as in $\S$ \ref{sub:Coordinate-presentation}, we let $\infty=\overline{f}_{*}(H)\in\mathbb{P}^{1}$. 
\begin{defn}
\label{def:good-resolution} A \emph{good resolution} of $\overline{f}$
is a triple $(\tilde{\mathbb{P}},\sigma,\tilde{f})$ consisting of
a projective threefold $\tilde{\mathbb{P}}$, a birational morphism
$\sigma:\tilde{\mathbb{P}}\rightarrow\mathbb{P}$ and a morphism $\tilde{f}:\tilde{\mathbb{P}}\rightarrow\mathbb{P}^{1}$
satisfying the following properties:

a) The diagram \[\xymatrix{ \tilde{\mathbb{P}} \ar[r]^{\sigma} \ar[d]_{\tilde{f}} & \mathbb{P} \ar@{-->}[d]^{\overline{f}} \\ \mathbb{P}^1 \ar@{=}[r] & \mathbb{P}^1 }\]
commutes. 

b) $\tilde{\mathbb{P}}$ has at most $\mathbb{Q}$-factorial terminal
singularities and is smooth outside $\tilde{f}^{-1}(\infty)$. 

c) $\sigma:\tilde{\mathbb{P}}\rightarrow\mathbb{P}$ is a sequence
of blow-ups whose successive centers lie above the base locus of $\mathcal{L}$,
inducing an isomorphism $\tilde{\mathbb{P}}\setminus\sigma^{-1}(H)\stackrel{\sim}{\rightarrow}\mathbb{P}\setminus H$,
and whose restriction to every closed fiber of $\tilde{f}$ except
$\tilde{f}^{-1}(\infty)$ is an isomorphism onto its image. 
\end{defn}

\subsubsection{\noindent }

\label{par:good-resolutuon-properties} It follows from the definition
that all irreducible divisors in the exceptional locus $\mathrm{Exc}(\sigma)$
of a good resolution $\sigma$ that are vertical for $\tilde{f}$
are contained in $\tilde{f}^{-1}(\infty)$. Furthermore, since the
restriction of $\sigma$ to the generic fiber of $\tilde{f}$ is an
isomorphism onto the generic member of $\mathcal{L}$, $\mathrm{Exc}(\sigma)$
contains exactly as many irreducible horizontal divisors as there
are irreducible components in $H\cap S$. Indeed, there is a one to
one correspondence between irreducible horizontal divisors in $\mathrm{Exc}(\sigma)$
and irreducible components of the intersection of $\sigma^{-1}(H)$
with the generic fiber of $\tilde{f}$. By assumption, the latter
is isomorphic to the smooth del Pezzo surface $S_{\eta}$ of degree
$d$ in $\mathbb{P}_{\mathbb{C}(\lambda)}$ with equation $s(x,y,z,w)-\lambda x^{e}=0$
(see $\S$ \ref{sub:Coordinate-presentation}), and the definition
of $(\tilde{\mathbb{P}},\sigma,\tilde{f})$ implies that it intersects
$\sigma^{-1}(H)$ along the curve $D_{\eta}\simeq(H\cap S)\times_{\mathrm{Spec}(\mathbb{C})}\mathrm{Spec}(\mathbb{C}(\lambda))$
with equation $s(0,y,z,w)=0$ in $\mathrm{Proj}(\mathbb{C}(\lambda)[y,z,w])$.
In particular, $D_{\eta}$ is an anti-canonical divisor on $S_{\eta}$
with the same number of irreducible components as $H\cap S$, all
them being defined over $\mathbb{C}(\lambda)$. Note also that the
intersection of $\sigma^{-1}(H)$ with a closed fiber $\tilde{f}^{-1}(c)$
distinct from $\tilde{f}^{-1}(\infty)$ is isomorphic to the intersection
of $H$ with the corresponding member $\sigma(\tilde{f}^{-1}(c))$
of $\mathcal{L}$.\\

A good resolution $(\tilde{\mathbb{P}},\sigma,\tilde{f})$ of $f:\mathbb{P}\dashrightarrow\mathbb{P}^{1}$
always exists. For instance, let $\tau:X\rightarrow\mathbb{P}$ be
the blow-up of scheme-theoretic base locus of $\mathcal{L}$. Then
$X$ is isomorphic to the hypersurface in $\mathbb{P}\times\mathrm{Proj}(\mathbb{C}[\alpha,\beta])$
defined by the weighted bi-homogeneous equation $\beta s(x,y,z,w)-\alpha\beta x^{e}=0$,
and we have a commutative diagram \[\xymatrix{X   \ar[dr]_{\pi=\mathrm{pr}_2\mid_X}  \ar[r]^{\tau}  & \mathbb{P} \ar@{-->}[d]^{\overline{f}} \\   & \mathbb{P}^1 }\]
The morphism $\tau$ restricts on each fiber of $\pi$ to an isomorphism
onto the corresponding member of $\mathcal{L}$ and $X\setminus\tau^{-1}(H)\simeq\mathbb{P}\setminus H$.
Furthermore, since $S$ is smooth, it follows from the Jacobian criterion
that $X$ is smooth outside $\pi^{-1}(\infty)$. Letting $\tau_{1}:\tilde{\mathbb{P}}\rightarrow X$
be any resolution of the singularities of $X$, the triple $(\tilde{\mathbb{P}},\tau\circ\tau_{1},\pi\circ\tau_{1})$
is a good resolution of $\overline{f}$ for which $\tilde{\mathbb{P}}$
is even smooth.

\subsection{\label{sub:Relative-MMP}Basic properties of relative MMPs ran from
good resolutions }

Let $(\tilde{\mathbb{P}},\sigma,\tilde{f})$ be a good resolution
of the rational map $\overline{f}:\mathbb{P}\dashrightarrow\mathbb{P}^{1}$
associated to a pencil $\mathcal{L}\subset|\mathcal{O}_{\mathbb{P}}(e)|$
as above. Recall \cite[3.31]{KM98} that a MMP $\varphi:\tilde{\mathbb{P}}_{0}=\tilde{\mathbb{P}}\dashrightarrow\tilde{\mathbb{P}}'=\tilde{\mathbb{P}}_{n}$
relative to $\tilde{f}_{0}=\tilde{f}:\tilde{\mathbb{P}}_{0}\rightarrow\mathbb{P}^{1}$
consists of a finite sequence $\varphi=\varphi_{n}\circ\cdots\circ\varphi_{1}$
of birational maps 
\begin{eqnarray*}
\tilde{\mathbb{P}}_{k-1} & \stackrel{\varphi_{k}}{\dashrightarrow} & \tilde{\mathbb{P}}_{k}\\
\tilde{f}_{k-1}\downarrow &  & \downarrow\tilde{f}_{k}\qquad k=1,\ldots,n,\\
\mathbb{P}^{1} & = & \mathbb{P}^{1}
\end{eqnarray*}
where each $\varphi_{k}$ is associated to an extremal ray $R_{k-1}$
of the closure $\overline{NE}(\tilde{\mathbb{P}}_{k-1}/\mathbb{P}_{1})$
of the relative cone of curves of $\tilde{\mathbb{P}}_{k-1}$ over
$\mathbb{P}^{1}$. Each of these birational maps $\varphi_{k}$ is
either of divisorial contraction or a flip whose flipping and flipped
curves are contained in the fibers of $\tilde{f}_{k-1}$ and $\tilde{f}_{k}$
respectively. Letting $\Delta_{0}=\sigma^{-1}(H)$ and $\Delta_{k}=({\varphi_{k}})_{*}(\Delta_{k-1})$
for every $k=1,\ldots,n$, the next result asserts in particular that
every relative MMP ran from a good resolution of $\overline{f}:\mathbb{P}\dashrightarrow\mathbb{P}^{1}$
preserves the open subset $\sigma^{-1}(\mathbb{P}\setminus H)\simeq\mathbb{P}\setminus H\simeq\mathbb{A}^{3}$. 
\begin{prop}
\label{prop:MMP-preserving} \label{lem:MMP-intermediate-steps} Let
$\mathcal{L}\subset|\mathcal{O}_{\mathbb{P}}(e)|$ be as above and
let $(\tilde{\mathbb{P}},\sigma,\tilde{f})$ be any good resolution
of the corresponding rational map $\overline{f}:\mathbb{P}\dashrightarrow\mathbb{P}^{1}$.
Then every MMP $\varphi:\tilde{\mathbb{P}}\dashrightarrow\tilde{\mathbb{P}}'$
relative to $\tilde{f}:\tilde{\mathbb{P}}\rightarrow\mathbb{P}^{1}$
restricts to an isomorphism $\mathbb{A}^{3}\simeq\tilde{\mathbb{P}}\setminus\sigma^{-1}(H)\stackrel{\sim}{\rightarrow}\tilde{\mathbb{P}}'\setminus\varphi_{*}(\sigma^{-1}(H))$.
More precisely, the following hold at each intermediate step: 

a) The threefold $\tilde{\mathbb{P}}_{k}$ is smooth outside $\tilde{f}_{k}^{-1}(\infty)$,

b) The birational map $\varphi_{k}:\tilde{\mathbb{P}}_{k-1}\dashrightarrow\tilde{\mathbb{P}}_{k}$
restricts to an isomorphism $\tilde{\mathbb{P}}_{k-1}\setminus\Delta_{k-1}\rightarrow\tilde{\mathbb{P}}_{k}\setminus\Delta_{k}$,

c) The restriction of $\varphi_{k}$ to a general closed fiber of
$\tilde{f}_{k}$ is either an isomorphism onto its image, or the contraction
of finitely many disjoint $(-1)$-curves. \end{prop}
\begin{proof}
Since by virtue of Lemma \ref{lem:irreducible-members}, all members
of $\mathcal{L}$ except $eH$ are irreducible and reduced, the fact
that $(\tilde{\mathbb{P}},\sigma,\tilde{f})$ is a good resolution
guarantees that all fibers of $\tilde{f}_{0}$ except maybe $\tilde{f}_{0}^{-1}(\infty)$
are irreducible and reduced. This implies in turn that the divisors
contracted by $\varphi:\tilde{\mathbb{P}}_{0}\dashrightarrow\tilde{\mathbb{P}}_{n}$
are either irreducible components of $\tilde{f}_{0}^{-1}(\infty)$
or are horizontal for $\tilde{f}_{0}$. Let $\varphi_{0}=\mathrm{id}_{\tilde{\mathbb{P}}_{0}}$.
If $\varphi_{k}$, $k\geq1,$ is the contraction of a divisor $E_{k-1}\subset\tilde{\mathbb{P}}_{k-1}$
onto a curve $B_{k}\subset\tilde{\mathbb{P}}_{k}$, then by the previous
observation, $E$ is either an irreducible component of $\tilde{f}_{k-1}^{-1}(\infty)$
or is horizontal for $\tilde{f}_{k-1}$. In the second case, $E_{k-1}$
is the proper transform in $\tilde{\mathbb{P}}_{k-1}$ of an irreducible
divisor $E\subset\tilde{\mathbb{P}}_{0}$, which is necessarily contained
in the support of $\Delta_{0}$. Indeed, by induction hypothesis,
the restriction $\varphi_{k-1}\circ\cdots\varphi_{1}\circ\varphi_{0}:S_{c,0}=\tilde{f}_{0}^{-1}(c)\rightarrow S_{c,k-1}=\tilde{f}_{k}^{-1}(c)$
to a general closed fiber of $\tilde{f}_{0}$ is either an isomorphism
or a sequence of contractions of $(-1)$-curves. Since $E_{k-1}\cap S_{c,k-1}$
consists of a disjoint union of $(-1)$-curves, it follows that $E\cap S_{c,0}$
is a curve $C$ on $S_{c,0}$ that can be contracted to a finite number
of smooth points, hence consists of a disjoint union of $(-1)$-curves
because $S_{c,0}$ is a smooth del Pezzo surface. But on the other
hand, if $E$ were not $\sigma$-exceptional, the hypothesis that
$\sigma$ maps $S_{c,0}$ isomorphically onto its image in $\mathbb{P}$
would imply that the proper transform $\sigma_{*}E$ of $E$ in $\mathbb{P}$
is an ample divisor intersecting $\sigma(S_{c,0})$ along the curve
$\sigma(C)$ which is absurd as $\sigma(C)$ consists again of a disjoint
union of $(-1)$-curves. Thus $E$ is contained in $\Delta_{0}$ and
hence $E_{k-1}$ is contained in $\Delta_{k-1}$. Furthermore, since
$\tilde{\mathbb{P}}_{k-1}\setminus\tilde{f}_{k-1}^{-1}(\infty)$ is
smooth by hypothesis, it follows that $\tilde{\mathbb{P}}_{k}\setminus\tilde{f}_{k}^{-1}(\infty)$
is still smooth along $B_{k}\setminus(B_{k}\cap\tilde{f}_{k}^{-1}(\infty))$.
More precisely, $B_{k}\setminus(B_{k}\cap\tilde{f}_{k}^{-1}(\infty))$
is smooth and 
\[
\varphi_{k}\mid_{\tilde{\mathbb{P}}_{k-1}\setminus\tilde{f}_{k-1}^{-1}(\infty)}:\tilde{\mathbb{P}}_{k-1}\setminus\tilde{f}_{k-1}^{-1}(\infty)\rightarrow\tilde{\mathbb{P}}_{k}\setminus\tilde{f}_{k}^{-1}(\infty)
\]
coincides with the blow-up of $\tilde{\mathbb{P}}_{k}\setminus\tilde{f}_{k}^{-1}(\infty)$
along $B_{k}\setminus(B_{k}\cap\tilde{f}_{k}^{-1}(\infty))$ \cite{Cut88}.
Finally, the restriction of $\varphi_{k}$ to a general closer fiber
of $\tilde{f}_{k-1}$ is either an isomorphism onto its image, or
the contraction of finitely many disjoint $(-1)$-curves, in particular
its image by $\varphi_{k}$ is again a smooth del Pezzo surface. Otherwise,
if $\varphi_{k}$ is a flip, then since its flipping curves must pass
through a singular point of $\tilde{\mathbb{P}}_{k-1}$ \cite[14.6.4]{CKM88},
they are contained in $\tilde{f}_{k-1}^{-1}(\infty)$. The flipped
curves of $\varphi_{k}$ are thus contained in $\tilde{f}_{k}^{-1}(\infty)$
and $\varphi_{k}$ restricts to an isomorphism between $\tilde{\mathbb{P}}_{k-1}\setminus\tilde{f}_{k-1}^{-1}(\infty)$
and $\tilde{\mathbb{P}}_{k}\setminus\tilde{f}_{k}^{-1}(\infty)$,
which is thus again smooth. 
\end{proof}

\section{outputs of relative MMPs}

Since a general member of a pencil $\mathcal{L}\subset|\mathcal{O}_{\mathbb{P}}(e)|$
as in Definition \ref{def:Pencils} above is a rational surface, the
output $\tilde{\mathbb{P}}'$ of a relative MMP $\varphi:\tilde{\mathbb{P}}\dashrightarrow\tilde{\mathbb{P}}'$
ran from a good resolution $(\tilde{\mathbb{P}},\sigma,\tilde{f})$
of the corresponding rational map $\overline{f}:\mathbb{P}\dashrightarrow\mathbb{P}^{1}$
is a Mori fiber space $\tilde{f}':\tilde{\mathbb{P}}'\rightarrow\mathbb{P}^{1}$.
More precisely, $\tilde{f}':\tilde{\mathbb{P}}'\rightarrow\mathbb{P}^{1}$
is either a del Pezzo fibration with relative Picard number $1$,
or a Mori conic bundle over a certain normal projective surface $W$,
say $\tilde{f}'=q\circ\xi:\tilde{\mathbb{P}}'\rightarrow W\rightarrow\mathbb{P}^{1}$
where $\xi:\tilde{\mathbb{P}}'\rightarrow W$ is a flat morphism of
relative Picard number $1$, with connected fibers and such that $-K_{\tilde{\mathbb{P}}'}$
is relatively ample. In each case, it follows from Proposition \ref{prop:MMP-preserving}
that $\tilde{\mathbb{P}}'$ is a projective completion of $\mathbb{A}^{3}$
with at most $\mathbb{Q}$-factorial terminal singularities. The following
theorem shows in particular that except maybe in the case where $d=3$
and $H\cap S$ consists of two irreducible components, the nature
of $\tilde{\mathbb{P}}'$ depends only on the base locus of $\mathcal{L}$.
In particular, it depends neither on the chosen good resolution $(\tilde{\mathbb{P}},\sigma,\tilde{f})$
nor on the relative MMP $\varphi:\tilde{\mathbb{P}}\dashrightarrow\tilde{\mathbb{P}}'$. 
\begin{thm}
\label{thm:MMP-outputs} Let \textup{$\mathcal{L}\subset|\mathcal{O}_{\mathbb{P}}(e)|$}
be the pencil generated by a smooth del Pezzo surface $S\subset\mathbb{P}$
of degree $d\in\{1,2,3\}$ and $H\in|\mathcal{O}_{\mathbb{P}}(1)|$,
let $(\tilde{\mathbb{P}},\sigma,\tilde{f})$ be a good resolution
of the corresponding rational map $\overline{f}:\mathbb{P}\dashrightarrow\mathbb{P}^{1}$,
and let $\varphi:\tilde{\mathbb{P}}\dashrightarrow\tilde{\mathbb{P}}'$
be a relative MMP. Then the following hold:

a) If $H\cap S$ is irreducible, then $\tilde{f}':\tilde{\mathbb{P}}'\rightarrow\mathbb{P}^{1}$
is a del Pezzo fibration of degree $d$. 

b) If $d=2$ and $H\cap S$ is reducible, then $\tilde{f}':\tilde{\mathbb{P}}'\rightarrow\mathbb{P}^{1}$
is a del Pezzo fibration of degree $d+1=3$. 

c) If $H\cap S$ has three irreducible components, then $\tilde{\mathbb{P}}'$
is a Mori conic bundle. \end{thm}
\begin{proof}
If $H\cap S$ is irreducible then $\sigma^{-1}(H)$ has a unique horizontal
irreducible component, whose intersection with the generic fiber $S_{\eta}$
of $\tilde{f}:\tilde{\mathbb{P}}\rightarrow\mathbb{P}^{1}$ is an
irreducible anti-canonical divisor with self-intersection $d$. So
with the notation of section \ref{sub:Relative-MMP} and Proposition
\ref{prop:MMP-preserving}, it follows that at each intermediate step
$\varphi_{k}:\tilde{\mathbb{P}}_{k-1}\dashrightarrow\tilde{\mathbb{P}}_{k}$
of $\varphi$, the intersection of $\Delta_{k-1}$ with the generic
fiber of $\tilde{f}_{k-1}:\tilde{\mathbb{P}}_{k-1}\rightarrow\mathbb{P}^{1}$
is an irreducible curve with non negative self-intersection, which
is therefore not contracted by $\varphi_{k}$. So $\varphi$ does
not contract the unique horizontal irreducible component of $\sigma^{-1}(H)$.
It follows that $\varphi$ restricts to an isomorphism between the
generic fibers of $\tilde{f}:\tilde{\mathbb{P}}\rightarrow\mathbb{P}^{1}$
and $\tilde{f}':\tilde{\mathbb{P}}'\rightarrow\mathbb{P}^{1}$, the
former being a smooth del Pezzo surface of degree $d$ over the function
field $\mathbb{C}(\lambda)$ of $\mathbb{P}^{1}$ by virtue of $\S$
\ref{par:good-resolutuon-properties}. On the other hand, Lemma \ref{lem:Mcb-characterization}
below implies that $\tilde{f}':\tilde{\mathbb{P}}'\rightarrow\mathbb{P}^{1}$
cannot be a Mori conic bundle, and so $\tilde{f}':\tilde{\mathbb{P}}'\rightarrow\mathbb{P}^{1}$
is a del Pezzo fibration of degree $d$. 

If $d=2$ and $H\cap S$ is reducible, then $\sigma^{-1}(H)$ consists
of two horizontal irreducible components, and its intersection with
the generic fiber $S_{\eta}$ of $\tilde{f}:\tilde{\mathbb{P}}\rightarrow\mathbb{P}^{1}$
is a reduced anti-canonical divisor whose support consists of the
union of two $(-1)$-curves $C_{1}$ and $C_{2}$ defined over $\mathbb{C}(\lambda)$
intersecting each other twice, either with multiplicity $2$ at a
unique $\mathbb{C}(\lambda)$-rational point, or transversally at
a pair of distinct $\mathbb{C}(\lambda)$-rational points, or at unique
point whose residue field is a quadratic extension of $\mathbb{C}(\lambda)$
(see \ref{par:good-resolutuon-properties}). These two curves being
independent in the Néron-Severi group of $S_{\eta}$, the Picard number
$\rho(S_{\eta})$ is bigger or equal to $2$. If $\varphi$ does not
contract any horizontal component of $\sigma^{-1}(H)$ then $\varphi$
restricts to an isomorphism between $S_{\eta}$ and the generic fiber
$S_{\eta}'$ of $\tilde{f}':\tilde{\mathbb{P}}'\rightarrow\mathbb{P}^{1}$.
Since $\rho(S_{\eta}')=\rho(S_{\eta})\geq2$, this implies that $\tilde{f}':\tilde{\mathbb{P}}'\rightarrow\mathbb{P}^{1}$
is a Mori conic bundle $\xi:\tilde{\mathbb{P}}'\rightarrow W$ over
a normal projective surface $q:W\rightarrow\mathbb{P}^{1}$. Furthermore,
the general fibers of $\tilde{f}'$ being rational, so are the general
fibers of $q$, implying that $q:W\rightarrow\mathbb{P}^{1}$ is a
$\mathbb{P}^{1}$-fibration. Restricting $\xi$ over the generic point
$\eta$ of $\mathbb{P}^{1}$, we obtain a Mori conic bundle $\xi_{\eta}:S'_{\eta}\rightarrow W_{\eta}\simeq\mathbb{P}_{\mathbb{C}(\lambda)}^{1}$
defined over $\mathbb{C}(\lambda)$. Letting $C_{1}'$ and $C_{2}'$
be the images of $C_{1}$ and $C_{2}$ respectively in $S_{\eta}'$,
we have $-K_{S_{\eta}'}\sim C_{1}'+C_{2}'$ and since $(-K_{S_{\eta}'}\cdot\ell)=2$
for every general $\mathbb{C}(\lambda)$-rational fiber $\ell$ of
$\xi_{\eta}$, it follows that either $C_{1}'$ and $C_{2}'$ are
both sections of $\xi_{\eta}$ or, up to a permutation, that $C_{1}'$
is a $2$-section of $\xi_{\eta}$ while $C_{2}'$ is contained in
a fiber. The second possibility is excluded because a Mori conic bundle
over $\mathbb{P}_{\mathbb{C}(\lambda)}^{1}$ does not contain any
$(-1)$-curve defined over $\mathbb{C}(\lambda)$ in its closed fibers.
In the first case, since the relative Picard number $\rho(S'_{\eta}/\mathbb{P}_{\mathbb{C}(\lambda)}^{1})$
is equal to $1$, we would have $C_{2}'\sim C_{1}'+a\ell$ for some
$a\in\mathbb{Q}$ such that $2=C_{1}'\cdot C_{2}'=(C_{1}')^{2}+a=-1+a$
and $-1=(C_{2}')^{2}=\left(C_{1}'\right)^{2}+2a=-1+2a$, which is
absurd. So $\varphi$ contracts at least one of the two horizontal
irreducible components of $\sigma^{-1}(H)$, say the one intersecting
$S_{\eta}$ along $C_{1}$. Letting $\varphi_{k}:\tilde{\mathbb{P}}_{k-1}\dashrightarrow\tilde{\mathbb{P}}_{k}$
be the intermediate step of $\varphi$ at which this contraction occurs,
the induced morphism $\varphi_{k,\eta}:S_{k-1,\eta}\rightarrow S_{k,\eta}$
between the generic fibers of $\tilde{f}_{k-1}:\tilde{\mathbb{P}}_{k-1}\rightarrow\mathbb{P}^{1}$
and $\tilde{f}_{k}:\tilde{\mathbb{P}}_{k}\rightarrow\mathbb{P}^{1}$
coincides with the contraction of $C_{1}$. So $S_{k,\eta}$ is a
smooth del Pezzo surface of degree $3$ defined over $\mathbb{C}(\lambda)$,
which intersects the proper transform $\Delta_{k}$ of $\sigma^{-1}(H)$
along the image of $C_{2}$. The latter being an irreducible $\mathbb{C}(\lambda)$-rational
curve with self-intersection $3$, the same argument as in the previous
case implies that the corresponding horizontal irreducible component
of $\Delta_{k}$ cannot be contracted at any further step $\varphi_{k'}$,
$k'\geq k+1$, of $\varphi$. So $\varphi$ contracts exactly one
irreducible component of $\sigma^{-1}(H)$ and the generic fiber $S_{\eta}'$
is isomorphic to the image of $S_{\eta}$ by the contraction of the
corresponding $(-1)$-curve defined over $\mathbb{C}(\lambda)$. Thus
$S_{\eta}'$ is a smooth del Pezzo surface of degree $3$ defined
over $\mathbb{C}(\lambda)$. We deduce again from Lemma \ref{lem:Mcb-characterization}
that $\tilde{f}':\tilde{\mathbb{P}}'\rightarrow\mathbb{P}^{1}$ cannot
be a Mori conic bundle, and so $\tilde{f}':\tilde{\mathbb{P}}'\rightarrow\mathbb{P}^{1}$
is a del Pezzo fibration of degree $3$. 

Finally, if $d=3$ and $H\cap S$ has three irreducible components,
then the intersection of $\sigma^{-1}(H)$ with $S_{\eta}$ is a reduced
anti-canonical divisor on $S_{\eta}$ whose support consists of the
union of three $(-1)$-curves $C_{1}$, $C_{2}$ and $C_{3}$ defined
over $\mathbb{C}(\lambda)$ and intersecting each other transversally
at $\mathbb{C}(\lambda)$-rational points. If $\varphi$ does not
contract any horizontal irreducible component of $\sigma^{-1}(H)$,
then it induces an isomorphism between $S_{\eta}$ and the generic
fiber $S_{\eta}'$ of $\tilde{f}':\tilde{\mathbb{P}}'\rightarrow\mathbb{P}^{1}$.
The latter is thus a smooth del Pezzo surface of degree $3$ defined
over $\mathbb{C}(\lambda)$ and having the sum $C_{1}'+C_{2}'+C_{3}'$
of the images of the $C_{i}$'s as an anti-canonical divisor. The
Picard number of $S_{\eta}'$ is thus strictly bigger than one, and
so $\tilde{f}':\tilde{\mathbb{P}}'\rightarrow\mathbb{P}^{1}$ is again
a Mori conic bundle, restricting over the generic point $\eta$ of
$\mathbb{P}^{1}$ to a Mori conic bundle $\xi_{\eta}:S_{\eta}'\rightarrow\mathbb{P}_{\mathbb{C}(\lambda)}^{1}$
defined over $\mathbb{C}(\lambda)$. Since $(-K_{S_{\eta}'}\cdot\ell)=2$
for every general $\mathbb{C}(\lambda)$-rational fiber $\ell$ of
$\xi_{\eta}$, either two of the $C_{i}'$ are sections of $\xi_{\eta}$
and the third one is contained in a fiber or one of the $C_{i}'$
is a $2$-section of $\xi_{\eta}$ and the two other ones are contained
in a fiber. In each case, there would exists a closed fiber of $\xi_{\eta}:S_{\eta}'\rightarrow\mathbb{P}_{\mathbb{C}(\lambda)}^{1}$
containing a $(-1)$-curve defined over $\mathbb{C}(\lambda)$, which
is impossible. So $\varphi$ contracts at least one horizontal irreducible
component of $\sigma^{-1}(H)$, say the one intersecting $S_{\eta}$
along $C_{1}$. The proper transforms of $C_{2}$ and $C_{3}$ in
the image of $S_{\eta}$ by the induced contraction are $0$-curves
intersecting each other twice at $\mathbb{C}(\lambda)$-rational points.
The same argument as in the previous case implies that no other horizontal
irreducible component of $\sigma^{-1}(H)$ is contracted by $\varphi$.
So $S_{\eta}'$ is isomorphic to the image of $S_{\eta}$ by the contraction
of $C_{1}$, hence is a smooth del Pezzo surface of degree $4$ defined
over $\mathbb{C}(\lambda)$, having the sum $C_{2}'+C_{3}'$ of the
images of $C_{2}$ and $C_{3}$ as an anti-canonical divisor. The
Picard number $\rho(S_{\eta}')$ is thus bigger or equal to $2$ and
so, $\tilde{f}':\tilde{\mathbb{P}}'\rightarrow\mathbb{P}^{1}$ is
necessarily a Mori conic bundle. 
\end{proof}
In the proof of Theorem \ref{thm:MMP-outputs} above, we used the
following criterion for the output of a relative MMP $\varphi:\tilde{\mathbb{P}}\dashrightarrow\tilde{\mathbb{P}}'$
to be a Mori conic bundle:
\begin{lem}
\label{lem:Mcb-characterization}With the notation above, let $r\in\left\{ 1,2,3\right\} $
and $h_{\varphi}\in\{0,1\}$ be the number of irreducible components
of $H\cap S$ and the number of horizontal irreducible component of
$\sigma^{-1}(H)$ contracted by $\varphi:\tilde{\mathbb{P}}\dashrightarrow\tilde{\mathbb{P}}'$.
If $\tilde{\mathbb{P}}'$ is a Mori conic bundle $\tilde{f}'=q\circ\xi:\tilde{\mathbb{P}}'\rightarrow W\rightarrow\mathbb{P}^{1}$,
then $r=h_{\varphi}+2$. \end{lem}
\begin{proof}
We first observe that the inverse image by $\xi$ of every irreducible
curve $C\subset W$ is again irreducible. Indeed, assuming on the
contrary that $\xi^{-1}(C)$ has at least two irreducible components
$F_{1}$ and $F_{2}$ such that $F_{1}\cap F_{2}\neq\emptyset$, we
can choose an irreducible curve $\ell_{1}\subset F_{1}$ whose class
$[\ell_{1}]$ in $\overline{NE}(\tilde{\mathbb{P}}')$ belongs to
the extremal ray giving rise to $\xi$ and such that $\ell_{1}\cap F_{2}\neq\emptyset$.
Then for a general fiber $\ell$ of $\xi$, we have by definition
$[\ell]=a[\ell_{1}]$ for some $a>0$, but since $\ell$ is disjoint
from $F_{2}$, this would lead to the contradiction $0=F_{2}\cdot\ell=aF_{2}\cdot\ell_{1}>0$.
Since all fibers of $\tilde{f}'$ except maybe $(\tilde{f}')^{-1}(\infty)$
are irreducible and rational, it follows that $q:W\rightarrow\mathbb{P}^{1}$
is a $\mathbb{P}^{1}$-fibration with $\eta^{-1}(\infty)$ as a unique
possibly reducible fiber. In particular, the Picard number $\rho(W)$
is equal to $\nu_{\infty}+1$, where $\nu_{\infty}$ denotes the number
of irreducible components of $\eta^{-1}(\infty)$, which by the previous
observation is equal to the number of irreducible components of $(\tilde{f}')^{-1}(\infty)$.
Since $(\tilde{\mathbb{P}},\sigma,\tilde{f})$ is a good resolution,
the number of horizontal irreducible components of $\sigma^{-1}(H)$
is equal to $r$. So the Picard number $\rho(\tilde{\mathbb{P}})$
of $\tilde{\mathbb{P}}$ is equal to $\rho(\mathbb{P})+r+e_{v}=1+r+e_{v}$,
where $e_{v}$ denote the number of vertical exceptional divisors
of $\sigma$, all of them being contained in $\tilde{f}^{-1}(\infty)$
(see $\S$ \ref{par:good-resolutuon-properties}). We obtain 
\[
\nu_{\infty}+1=\rho(W)=\rho(\tilde{\mathbb{P}}')-1=1+r+e_{v}-h_{\varphi}-v_{\varphi}-1=(1+e_{v}-v_{\varphi})+(r-h_{\varphi})-1=\nu_{\infty}+(r-h_{\varphi})-1
\]
where $v_{\varphi}$ denotes the number of vertical component of $\sigma^{-1}(H)$
contracted by $\varphi$. So $r=h_{\varphi}+2$. 
\end{proof}
\begin{parn}  \label{par:cubic-two-components} The remaining case
where $d=3$ and $H\cap S$ has two irreducible components is more
intricate. Here given a good resolution $(\tilde{\mathbb{P}},\sigma,\tilde{f})$
of the rational map $\overline{f}:\mathbb{P}=\mathbb{P}^{3}\dashrightarrow\mathbb{P}^{1}$,
the intersection of $\sigma^{-1}(H)$ with the generic fiber $S_{\eta}$
of $\tilde{f}:\tilde{\mathbb{P}}\rightarrow\mathbb{P}^{1}$ is a reduced
anti-canonical divisor whose support consists of the union of a $(-1)$-curve
$C_{1}$ and of a $0$-curve $C_{2}$ both defined over $\mathbb{C}(\lambda)$.
The same argument as in the proof of Theorem \ref{thm:MMP-outputs}
for the case $d=2$ with $H\cap S$ reducible implies that a relative
MMP $\varphi:\tilde{\mathbb{P}}\dashrightarrow\tilde{\mathbb{P}}'$
can contract at most one horizontal component of $\sigma^{-1}(H)$,
namely the one intersecting $S_{\eta}$ along $C_{1}$. If this component
is indeed contracted by $\varphi$, then the image of $S_{\eta}$
by the induced birational morphism is a smooth del Pezzo surface of
degree $4$ defined over $\mathbb{C}(\lambda)$ and the output $\tilde{f}':\tilde{\mathbb{P}}'\rightarrow\mathbb{P}^{1}$
is a del Pezzo fibration of degree $4$ by virtue of Lemma \ref{lem:Mcb-characterization}.
Otherwise, if $\varphi$ does not contract any horizontal irreducible
component of $\sigma^{-1}(H)$ then $\varphi$ restricts to an isomorphism
between $S_{\eta}$ and the generic fiber $S_{\eta}'$ of $\tilde{f}':\tilde{\mathbb{P}}'\rightarrow\mathbb{P}^{1}$.
Since $C_{1}+C_{2}$ is an anti-canonical divisor on $S_{\eta}$,
$\rho(S_{\eta})\geq2$ and so $\tilde{f}':\tilde{\mathbb{P}}'\rightarrow\mathbb{P}^{1}$
is necessarily a Mori conic bundle $\xi:\tilde{\mathbb{P}}'\rightarrow W$
over a normal projective surface $q:W\rightarrow\mathbb{P}^{1}$,
whose restriction over the generic point $\eta$ of $\mathbb{P}^{1}$
is a Mori conic bundle $\xi_{\eta}:S_{\eta}'\rightarrow W_{\eta}\simeq\mathbb{P}_{\mathbb{C}(\lambda)}^{1}$
defined over $\mathbb{C}(\lambda)$. Since $(-K_{S_{\eta}'}\cdot\ell)=2$
for every general $\mathbb{C}(\lambda)$-rational fiber $\ell$ of
$\xi_{\eta}$ and $C_{1}$ is a $(-1)$-curve defined over $\mathbb{C}(\lambda)$,
hence cannot be contained in a fiber of $\xi_{\eta}$, the only possibilities
are that either $C_{1}$ and $C_{2}$ are both sections of $\xi_{\eta}$
or that $C_{1}$ is a $2$-section of $\xi_{\eta}$ while $C_{2}$
is a full fiber of it. Similarly as in the case $d=2$ in the proof
of Theorem \ref{thm:MMP-outputs} above, the first possibility is
excluded by the fact that $\rho(S_{\eta}'/\mathbb{P}_{\mathbb{C}(\lambda)}^{1})=1$:
indeed, we would have $C_{2}\sim C_{1}+a\ell$ for some $a\in\mathbb{Q}$
satisfying simultaneously the identities $0=C_{2}^{2}=C_{1}^{2}+2a=-1+2a$
and $2=C_{2}\cdot C_{1}=C_{1}^{2}+a=-1+a$, which is impossible. But
in contrast with the case $d=2$, the second possibility cannot be
excluded. Actually a smooth cubic surface $S_{\eta}'\subset\mathbb{P}_{\mathbb{C}(\lambda)}^{3}$
containing a $(-1)$-curve $C_{1}$ defined over $\mathbb{C}(\lambda)$
always admit a conic bundle structure $\pi:S_{\eta'}\rightarrow\mathbb{P}_{\mathbb{C}(\lambda)}^{1}$
with five degenerate fibers, defined by the mobile part of the restriction
to $S'_{\eta}$ of the pencil of hyperplanes in $\mathbb{P}_{\mathbb{C}(\lambda)}^{3}$
containing $C_{1}$. 

\end{parn}

So in contrast with the other cases, this suggests that the nature
of the output $\tilde{\mathbb{P}}'$ might depend on the chosen good
resolution $(\tilde{\mathbb{P}},\sigma,\tilde{f})$ and on the relative
MMP $\varphi:\tilde{\mathbb{P}}\dashrightarrow\tilde{\mathbb{P}}'$.
Partial results on the structure of $\tilde{\mathbb{P}}'$ can be
obtained by a more careful study of relative MMPs ran from particular
explicit good resolutions $(\tilde{\mathbb{P}},\sigma,\tilde{f})$,
but a complete discussion would lead us far beyond the intended aim
of this article. The following result, which we mention without proof
referring the reader to the forthcoming paper \cite{DK16} for the
detail, asserts the existence of relative MMPs whose outputs are del
Pezzo fibrations of degre $4$. In contrast, we do not know examples
for which the output is a Mori conic bundle (see also Remark \ref{Rq:MCb-Degree-3}
below). 
\begin{prop}
\label{prop:cubic-dP4-termination} Let $S\subset\mathbb{P}^{3}$
be a smooth cubic surface, let $H\in\left|\mathcal{O}_{\mathbb{P}^{3}}(1)\right|$
be a hyperplane intersecting $S$ along the union of a line and smooth
conic, let $\mathcal{L}\subset\left|\mathcal{O}_{\mathbb{P}^{3}}(3)\right|$
be the pencil generated by $S$ and $3H$ and let $\overline{f}:\mathbb{P}^{3}\dashrightarrow\mathbb{P}^{1}$
be the corresponding rational map. Then there exists a good resolution
$(\tilde{\mathbb{P}},\sigma,\tilde{f})$ and a MMP $\varphi:\tilde{\mathbb{P}}\dashrightarrow\tilde{\mathbb{P}}'$
relative to $\tilde{f}:\tilde{\mathbb{P}}\rightarrow\mathbb{P}^{1}$
whose output is a del Pezzo fibration $\tilde{f}':\tilde{\mathbb{P}}'\rightarrow\mathbb{P}^{1}$
of degree $4$. 
\end{prop}

\section{Mori conic bundles and twisted $\mathbb{A}_{*}^{1}$-fibrations }

In this section, we investigate more closely the case where a relative
MMP $\varphi:\tilde{\mathbb{P}}\dashrightarrow\tilde{\mathbb{P}}'$
ran from a good resolution $(\tilde{\mathbb{P}},\sigma,\tilde{f})$
terminates with a Mori conic bundle $\xi:\tilde{\mathbb{P}}'\rightarrow W$
over a normal projective surface $W$. According to Theorem \ref{thm:MMP-outputs}
and $\S$ \ref{par:cubic-two-components},  this occurs for all pencils
$\mathcal{L}\subset\left|\mathcal{O}_{\mathbb{P}^{3}}(3)\right|$
generated by a smooth cubic surface $S\subset\mathbb{P}^{3}$ and
three times a hyperplane $H\subset\mathbb{P}^{3}$ such that $H\cap S$
consists of three lines, and possibly for pencils for which $H\cap S$
consists of a line and smooth conic when $\varphi:\tilde{\mathbb{P}}\dashrightarrow\tilde{\mathbb{P}}'$
does not contract any horizontal irreducible component of $\sigma^{-1}(H)$. 
\begin{thm}
\label{thm:Mcb-twisted-fibrations} Let $\mathcal{L}\subset|\mathcal{O}_{\mathbb{P}^{3}}(3)|$
be a pencil as above and let $\varphi:\tilde{\mathbb{P}}\dashrightarrow\tilde{\mathbb{P}}'$
be a relative MMP ran from good resolution $(\tilde{\mathbb{P}},\sigma,\tilde{f})$
of the corresponding rational map $\overline{f}:\mathbb{P}^{3}\dashrightarrow\mathbb{P}^{1}$
whose output is a Mori conic bundle $\xi:\tilde{\mathbb{P}}'\rightarrow W$
over a normal projective surface $q:W\rightarrow\mathbb{P}^{1}$.
Then there exists an open subset $U\subset W$ isomorphic to $\mathbb{A}^{2}$
such that the induced morphism $\xi_{0}=\xi\circ\varphi\circ\sigma^{-1}:\mathbb{A}^{3}=\mathbb{P}^{3}\setminus H\rightarrow W$
factors through a twisted $\mathbb{A}_{*}^{1}$-fibration over $U$. \end{thm}
\begin{proof}
Recall that by virtue of Proposition \ref{prop:MMP-preserving}, the
composition $\varphi\circ\sigma^{-1}:\mathbb{P}^{3}\setminus H\rightarrow\tilde{\mathbb{P}}'\setminus\varphi_{*}(\sigma^{-1}(H))$
is an isomorphism. As observed in the proof of Lemma \ref{lem:Mcb-characterization},
$q:W\rightarrow\mathbb{P}^{1}$ is a $\mathbb{P}^{1}$-fibration with
$\eta^{-1}(\infty)$ as a unique possibly reducible fiber, where $\infty=\overline{f}_{*}(H)$.
So the restriction of $q$ over $\mathbb{P}^{1}\setminus\{\infty\}$
is isomorphic to the trivial bundle $\mathbb{P}^{1}\setminus\{\infty\}\times\mathbb{P}^{1}$.
The union of all vertical components of $\varphi_{*}(\sigma^{-1}(H))$
is equal to $(\tilde{f}')^{-1}(\infty)$ (see $\S$ \ref{par:good-resolutuon-properties})
and on the other hand, it follows from the proof of Theorem \ref{thm:MMP-outputs}
and $\S$ \ref{par:cubic-two-components} that the restrictions of
the two horizontal irreducible components $E_{1}$ and $E_{2}$ of
$\varphi_{*}(\sigma^{-1}(H))$ to the generic fiber $S_{\eta}'$ of
$\tilde{f}'$ are either a pair of $0$-curves $C_{1}$ and $C_{2}$
defined over $\mathbb{C}(\lambda)$ with intersecting each other twice
at $\mathbb{C}(\lambda)$-rational points if $H\cap S$ consist of
three irreducible components, or the union of a $(-1)$-curve $C_{1}$
and a $0$-curve $C_{2}$ defined over $\mathbb{C}(\lambda)$ with
$(C_{1}\cdot C_{2})=2$ in the case where $H\cap S$ consists of two
irreducible components. In the first case, one of the curves $C_{i}$
is a $2$-section of the induced conic bundle $\xi_{\eta}:S_{\eta}'\rightarrow W_{\eta}\simeq\mathbb{P}_{\mathbb{C}(\lambda)}$
while the other one is a full fiber of it, and in the second case,
$C_{1}$ is a $2$-section of $\xi_{\eta}$ while $C_{2}$ is a full
fiber. So up to a permutation, we may assume that in both cases, $E_{1}$
is a birational $2$-section of $\xi:\tilde{\mathbb{P}}'\rightarrow W$
while $E_{2}$ is mapped by $\xi$ onto a section $D$ of $q:W\rightarrow\mathbb{P}^{1}$.
The open subset $U=W\setminus\xi(E_{2}\cup(\tilde{f}')^{-1}(\infty))=W\setminus(D\cup\eta^{-1}(\infty))$
of $W$ is thus isomorphic to $\mathbb{A}^{2}$, and by construction,
the composition $\xi_{0}=\xi\circ\varphi\circ\sigma^{-1}:\mathbb{A}^{3}=\mathbb{P}^{3}\setminus H\rightarrow W$
factors through $U$. Since $E_{1}$ is an irreducible birational
$2$-section of the conic bundle $\xi:\tilde{\mathbb{P}}'\rightarrow W$,
the generic fiber of $\xi_{0}$ is a nontrivial form of the punctured
affine line over the function field of $W$, so $\xi_{0}:\mathbb{A}^{3}\rightarrow U$
is a twisted $\mathbb{A}_{*}^{1}$-fibration. 
\end{proof}
\begin{parn} The twisted $\mathbb{A}_{*}^{1}$-fibrations $\xi_{0}:\mathbb{A}^{3}\rightarrow\mathbb{A}^{2}$
obtained in Theorem \ref{thm:Mcb-twisted-fibrations} above can be
described in terms of the initial data consisting of the smooth cubic
surface $S\subset\mathbb{P}^{3}$ and the hyperplane $H\in|\mathcal{O}_{\mathbb{P}^{3}}(1)|$
as follows. 

a) In the case where $H\cap S$ consists of the union of three lines
$\ell_{1}$, $\ell_{2}$ and $\ell_{3}$, then given a good resolution
$(\tilde{\mathbb{P}},\sigma,\tilde{f})$ of $\overline{f}:\mathbb{P}^{3}\dashrightarrow\mathbb{P}^{1}$,
the fiber of $\tilde{f}:\tilde{\mathbb{P}}\rightarrow\mathbb{P}^{1}$
over the generic point $\eta$ of $\mathbb{P}^{1}$ is a smooth cubic
surface $S_{\eta}\subset\mathbb{P}_{\mathbb{C}(\lambda)}^{3}$ defined
over $\mathbb{C}(\lambda)$ and the horizontal irreducible components
$E_{1}$, $E_{2}$ and $E_{3}$ of $\sigma^{-1}(H)$, corresponding
respectively to $\ell_{1}$, $\ell_{2}$ and $\ell_{3}$ intersect
$S_{\eta}$ along three $(-1)$-curves defined over $\mathbb{C}(\lambda)$.
For a relative MMP $\varphi:\tilde{\mathbb{P}}\dashrightarrow\tilde{\mathbb{P}}'$,
it follows from the description given in the proof of Theorem \ref{thm:MMP-outputs}
that exactly one horizontal irreducible component of $\sigma^{-1}(H)$
is contracted by $\varphi$, say $E_{3}$ up to a permutation. The
intersection of the proper transforms $\varphi_{*}(E_{1})$ and $\varphi_{*}(E_{2})$
of $E_{1}$ and $E_{2}$ with the generic fiber $S_{\eta'}$ of $\tilde{f}':\tilde{\mathbb{P}}'\rightarrow\mathbb{P}^{1}$
are $0$-curves defined over $\mathbb{C}(\lambda)$ intersecting each
other twice at $\mathbb{C}(\lambda)$-rational points. Furthermore,
one of them, say $\varphi_{*}(E_{2})\mid_{S_{\eta}'}$ is a fiber
of the induced Mori conic bundle structure $\xi_{\eta}:S_{\eta}'\rightarrow W_{\eta}\simeq\mathbb{P}_{\mathbb{C}(\lambda)}^{1}$,
the other one $\varphi_{*}(E_{1})\mid_{S_{\eta'}}$ being a $2$-section
of $\xi_{\eta}$. Therefore $\xi_{\eta}$ coincides with the proper
transform by the restriction $\varphi_{\eta}$ of $\varphi$ of the
conic bundle $\theta:S_{\eta}\rightarrow\mathbb{P}_{\mathbb{C}(\lambda)}^{1}$
defined by the mobile part of the restriction to $S{}_{\eta}$ of
the pencil of hyperplanes in $\mathbb{P}_{\mathbb{C}(\lambda)}^{3}$
containing $E_{1}\mid_{S_{\eta}}$. So letting $\Theta_{\ell_{1}}:\mathbb{P}^{3}\dashrightarrow\mathbb{P}^{1}$
be the projection from the line $\ell_{1}\subset H\cap S$, we conclude
that $\xi_{0}:\mathbb{A}^{3}=\mathbb{P}^{3}\setminus H\rightarrow\mathbb{A}^{2}$
coincides with the restriction to $\mathbb{P}^{3}\setminus H$ of
the rational map $\overline{f}\times\Theta_{\ell_{1}}:\mathbb{P}^{3}\dashrightarrow\mathbb{P}^{1}\times\mathbb{P}^{1}$. 

b) In the case where $H\cap S$ consists of the union of a line $\ell$
and a smooth conic, the description given in $\S$ \ref{par:cubic-two-components}
implies by a similar argument that $\xi_{0}:\mathbb{A}^{3}=\mathbb{P}^{3}\setminus H\rightarrow\mathbb{A}^{2}$
coincides with the restriction to $\mathbb{P}^{3}\setminus H$ of
the rational map $\overline{f}\times\Theta_{\ell}:\mathbb{P}^{3}\dashrightarrow\mathbb{P}^{1}\times\mathbb{P}^{1}$
where $\Theta_{\ell}:\mathbb{P}^{3}\dashrightarrow\mathbb{P}^{1}$
denotes the projection from the line $\ell$. 

\end{parn}
\begin{example}
\label{ex:twisted-example} Let $S\subset\mathbb{P}^{3}=\mathrm{Proj}_{\mathbb{C}}(\mathbb{C}[x,y,z,w])$
be the smooth cubic surface defined by the vanishing of the polynomial
$F=w^{2}z+y^{2}x+wx^{2}+z^{3}$, let $\overline{f}:\mathbb{P}^{3}\dashrightarrow\mathbb{P}^{1}$
be the pencil generated by $S$ and $3H$, where $H=\{x=0\}$ and
let 
\[
f:\mathbb{A}^{3}=\mathbb{P}^{3}\setminus H\simeq\mathrm{Spec}(\mathbb{C}[y,z,w])\rightarrow\mathbb{A}^{1},\;(y,z,w)\mapsto w^{2}z+y^{2}+w+z^{3}
\]
be the induced morphism. The intersection $H\cap S$ consists of three
lines $\ell_{1}=\{z=t=0\}$, $\ell_{2}=\{w+iz=t=0\}$ and $\ell_{3}=\{w-iz=t=0\}$
meeting in the Eckardt point $[0:1:0:0]$ of $S$, and the morphism
$\xi_{0}=(f,\mathrm{pr}_{z}):\mathbb{A}^{3}\rightarrow\mathbb{A}^{2}$
is a surjective twisted $\mathbb{A}_{*}^{1}$-fibration induced by
the restriction of $\overline{f}\times\Theta_{\ell_{1}}:\mathbb{P}^{3}\dashrightarrow\mathbb{P}^{1}\times\mathbb{P}^{1}$.
The fact that $\xi_{0}$ is twisted can be seen directly as follows:
its generic fiber is isomorphic to the curve $C\subset\mathbb{A}_{\mathbb{C}(\lambda,z)}^{2}=\mathrm{Spec}(\mathbb{C}(\lambda,z)[y,w])$
defined by the equation $w^{2}z+y^{2}+w+z^{3}-\lambda=0$. Extending
the scalars to the quadratic extension $K=\mathbb{C}(\lambda,z)[v]/(v^{2}-z)$,
we have 
\begin{eqnarray*}
C_{K} & \simeq & \mathrm{Spec}(K[x,y]/(w^{2}v^{2}+y^{2}+w+v^{6}-\lambda)\\
 & \simeq & \mathrm{Spec}(K[x,y]/((wv+\frac{1}{2v})^{2}+y^{2}-(\frac{1}{4v^{2}}-v^{6}+\lambda))\\
 & \simeq & \mathrm{Spec}(K[U,V]/(UV-(\frac{1}{4v^{2}}-v^{6}+\lambda))\\
 & \simeq & \mathrm{Spec}(K[U^{\pm1}])
\end{eqnarray*}
where $U=wv+\frac{1}{2v}+iy$ and $V=wv+\frac{1}{2v}-iy$, on which
the Galois group $\mathrm{Gal}(K/\mathbb{C}(\lambda,z))\simeq\mathbb{Z}_{2}$
acts by $U\mapsto-U^{-1}$. So $C$ is a nontrivial $\mathbb{C}(\lambda,z)$-form
of the punctured affine line over $\mathbb{C}(\lambda,z)$. \end{example}
\begin{rem}
\label{Rq:MCb-Degree-3} In the case where $d=3$ and $H\cap S$ consists
of a line $\ell$ and smooth conic, the fact that the projection $\Theta_{\ell}:\mathbb{P}^{3}\dashrightarrow\mathbb{P}^{1}$
gives rise to a twisted $\mathbb{A}_{*}^{1}$-fibration $\xi_{0}=(\overline{f},\Theta_{\ell})\mid_{\mathbb{P}^{3}\setminus H}:\mathbb{A}^{3}=\mathbb{P}^{3}\setminus H\rightarrow\mathbb{A}^{2}$
does not necessarily imply that a relative MMP $\varphi:\tilde{\mathbb{P}}\dashrightarrow\tilde{\mathbb{P}}'$
ran from a good resolution $(\tilde{\mathbb{P}},\sigma,\tilde{f})$
of $\overline{f}:\mathbb{P}^{3}\dashrightarrow\mathbb{P}^{1}$ terminates
with a Mori conic bundle $\xi:\tilde{\mathbb{P}}'\rightarrow W$ inducing
$\xi_{0}$ (see Proposition \ref{prop:cubic-dP4-termination}). Note
that since the base locus of $\Theta_{\ell}$ is contained in that
of $\overline{f}$, we can choose a good resolution $(\tilde{\mathbb{P}},\sigma,\tilde{f})$
of $\overline{f}$ which simultaneously resolves the indeterminacies
of $\Theta_{\ell}$. Every MMP $\psi:\tilde{\mathbb{P}}\dashrightarrow\tilde{\mathbb{P}}_{1}$
relative to the morphism $(\tilde{f},\Theta_{\ell}\circ\sigma):\tilde{\mathbb{P}}\rightarrow\mathbb{P}^{1}\times\mathbb{P}^{1}$
being also a part of a MMP relative to $\tilde{f}:\tilde{\mathbb{P}}\rightarrow\mathbb{P}^{1}$,
it preserves the open subset $\mathbb{A}^{3}=\tilde{\mathbb{P}}\setminus\sigma^{-1}(H)$
by virtue of Proposition \ref{prop:MMP-preserving}. Such a MMP process
$\psi$ does not contract any horizontal irreducible component of
$\sigma^{-1}(H)$ and terminates with a Mori conic bundle $\xi_{1}:\tilde{\mathbb{P}}_{1}\rightarrow\mathbb{P}^{1}\times\mathbb{P}^{1}$,
whose restriction to $\mathbb{A}^{3}$ coincides with $\xi_{0}$ by
construction. But there is no guarantee in general that $\tilde{f}_{1}=\mathrm{pr}_{1}\circ\xi_{1}:\tilde{\mathbb{P}}_{1}\rightarrow\mathbb{P}^{1}$
coincides with the final output of a MMP relative to $\tilde{f}:\tilde{\mathbb{P}}\rightarrow\mathbb{P}^{1}$:
there could exist a relative MMP $\varphi:\tilde{\mathbb{P}}\dashrightarrow\tilde{\mathbb{P}}'$
which factorizes through $\psi$ and for which the induced rational
map $\psi'=\varphi\circ\psi^{-1}:\tilde{\mathbb{P}}_{1}\dashrightarrow\tilde{\mathbb{P}}'$
contracts an irreducible component of $\psi_{*}(\sigma^{-1}(H))$
that is horizontal for $\tilde{f}_{1}$. 
\end{rem}
\bibliographystyle{amsplain}

\end{document}